\documentclass[12pt, reqno]{amsart}

\usepackage{amsmath, amsthm, amscd, amsfonts, amssymb, graphicx, color}
\usepackage[bookmarksnumbered, colorlinks, plainpages]{hyperref}

 \makeatletter \oddsidemargin.9375in
\evensidemargin \oddsidemargin \marginparwidth 2in \makeatother

\newtheorem{theorem}{Theorem}[section]
\newtheorem{cor}[theorem]{Corollary}
\newtheorem{lem}[theorem]{Lemma}

\newtheorem{defn}[theorem]{Definition}

\newtheorem{exam}[theorem]{Example}
\newtheorem{remark}[theorem]{Remark}
\numberwithin{equation}{section}

\newcommand{\BC}{{\Bbb C}}
\newcommand{\BN}{{\Bbb  N}}
\newcommand{\BR}{{\Bbb  R}}

\newcommand{\BK}{{\Bbb K}}

\newcommand{\Lip}{{\rm Lip}}
\newcommand{\lip}{{\rm lip}}
\newcommand{\coz}{{\rm coz}}

\newcommand{\supp}{{\rm supp}}

\newcommand{\A}{{\mathcal{A}}}

\newcommand{\lo}{{\longrightarrow}}

\newcommand{\B}{{\mathcal{B}}}

\begin{document}
\setcounter{page}{1}

\title[]{Additive jointly separating maps and ring homomorphisms}

\author[Fereshteh Sady and  Masoumeh Najafi Tavani]{Fereshteh  Sady and Masoumeh Najafi Tavani}

 \address{Department of Pure Mathematics, Faculty of Mathematical Sciences, Tarbiat Modares University, Tehran 14115--134, Iran}
 \email{\textcolor[rgb]{0.00,0.00,0.84}{sady@modares.ac.ir}}

\address{Department of Mathematics, Faculty of Basic Sciences, Islamic Azad University, Islamshahr Branch, Islamshahr, Tehran, Iran}
\email{\textcolor[rgb]{0.00,0.00,0.84}{najafi@iiau.ac.ir}}




\begin{abstract}
Let $X$ and $Y$ be compact Hausdorff spaces,  $E$ and $F$ be real or complex
normed spaces and $A(X,E)$ be a subspace of $C(X,E)$. For a
function $f\in C(X,E)$, let $\coz(f)$ be the cozero set of $f$. A
pair of additive maps $S,T: A(X,E) \lo C(Y,F)$ is said to be
jointly separating if $\coz(Tf)\cap \coz(Sg)=\emptyset$ whenever
$\coz(f)\cap \coz(g)= \emptyset$. In this paper, first we give a partial
description of additive jointly separating maps between certain
spaces of vector-valued continuous functions (including spaces of
vector-valued Lipschitz functions, absolutely continuous functions
and continuously differentiable functions). Then we apply the
results to characterize continuous ring homomorphisms between
certain Banach algebras of vector-valued continuous functions. In particular, the results provide some generalizations of the recent results on unital homomorphisms between vector-valued Lipschitz algebras, with a different approach.
\end{abstract}
\maketitle

\section{Introduction}
For a compact Hausdorff space $X$ and a real or complex Banach space  $E$, let $C(X,E)$ be the Banach space
of all continuous $E$-valued functions on $X$ under the supremum norm $\|f\|_\infty= \sup_{x\in X} \|f(x)\|$. In the scalar cases, that is the cases where $E=\Bbb C$, respectively  $E=\Bbb R$, this space is denoted by $C(X)$, respectively $C_\BR(X)$.

For compact Hausdorff spaces $X$ and $Y$ and real or complex Banach spaces $E$ and $F$, an additive  map $T$ from a subspace $A(X,E)$ to $C(Y, F)$ is said to be {\em separating} or {\em disjointness preserving} if for any pair $f,g$ of elements of $A(X,E)$ with disjoint cozeros, their images $Tf$ and $Tg$ also have disjoint cozeros.
Clearly in the scalar case, $T$ is separating if and only if it is zero product preserving, in the sense that $f\cdot g=0$ implies that $Tf\cdot Tg=0$. In particular, in this case, all (ring) homomorphisms between subalgebras of $C(X)$ (as well as $C_\BR(X)$) are separating.

Weighted composition operators are standard examples of separating maps between spaces of functions. More generally, if $X, Y$ are compact Hausdorff spaces, $E,F$ are Banach spaces and $A(X,E)$ and $A(Y,F)$ are subspaces of $C(X,E)$ and $C(Y,F)$, respectively, then any additive map $T:A(X,E)\longrightarrow A(Y,F)$ of the form
\[ Tf(y)=J_y(f(\varphi(y)))\;\;\;\;\;\;\;\;(f\in A(X,E))\]
where $\varphi:Y\longrightarrow X$ is a continuous map and $\{J_y\}_{y\in Y}$ is a family of additive maps from $E$ to $F$ is a
separating map.

Linear separating maps between  various spaces of continuous functions (in either of  scalar or vector valued case) have been studied for many years, see for instance \cite{bnt,fh,f,j,jw}. In most cases, it is shown that for certain subspaces of continuous functions any continuous linear separating map is of the above form, which is called a generalized weighted composition operator.
Continuous bilinear maps $\phi$ from $C^1[0,1] \times C^1[0,1]$
to a Banach space $E$ such that $f\cdot g=0$ implies $\phi(f,g)=0$ have been studied in  \cite{abcev}.  A similar problem has been considered in  \cite{aev} for the Banach algebra of  little Lipschitz functions instead of $C^1([0,1])$. Clearly any separating map $T: \mathcal A \lo \mathcal B$  between spaces of scalar-valued functions  $\mathcal A$ and $\mathcal B$, satisfies this implication for $\phi(f,g)=Tf \cdot Tg$.

Linear separating maps between vector-valued function spaces have been considered, for instance, in \cite{a,gjw, hbn,vw}.
Linear separating bijections  between spaces of vector-valued
continuous functions whose inverses are also separating , were studied in \cite{gjw} and  it was shown
 that such a map induces a homeomorphism between the underlying topological spaces. Similar results were
 given in \cite{vw} for such maps between vector-valued little Lipschitz function spaces.
In \cite{du}, Dubarbie studied
separating linear bijections on spaces of vector-valued absolutely continuous functions defined on compact
subsets of the real line.

On the other hand, in \cite{bot}, Botelho and Jamison  characterized  unital homomorphisms between vector-valued Banach algebra of Lipschitz functions with values in the Banach algebras $c$  and $\l_\infty$, as generalized weighted composition operators.
Then in \cite{Oi}, Oi extended  the result for $C(K)$-valued Lipschitz algebras, where  $K$ is a compact Hausdorff space. More general cases of unital homomorphisms between certain  Banach algebra-valued continuous functions have been studied recently in \cite{Hat}.

In this paper we first study a pair $T,S$ of additive maps between certain spaces of vector-valued continuous functions (on a compact Hausdorff space) which jointly preserve disjointness of cozero sets of functions. The results can be applied, for instance,  whenever $T$ and $S$ are defined between vector-valued Lipschitz functions, absolutely continuous functions and ($n$-times) continuously differentiable functions on the unit interval. Then considering the case that the target spaces are Banach algebras, we give some results concerning continuous unital ring homomorphisms between some vector-valued algebras  of functions. Hence, the results provide some generalizations of the results of \cite{bot} and \cite{Oi} with a different approach.

\section{ Main results}
We use the notation $\Bbb K$ for the field of real or complex
numbers. Let $X$  be a compact  Hausdorff space and $E$ be a Banach space
over $\Bbb K$. For $f\in C(X,E)$, we denote  the cozero set of $f$ by
$\coz(f)$, that is $\coz(f)= \{x\in X: f(x)\neq 0\}$. By a {\em
constant function} in $C(X,E)$ we mean a function on $X$ sending
all points of  $X$ to a fixed element $e \in E$. For $e\in E$ we
denote its corresponding constant function by $c_e$. For a subspace $A(X,E)$ of $C(X,E)$ and $x\in X$, the map $\delta_x: A(X,E) \longrightarrow E$ is defined by $\delta_x(f)=f(x)$.
For  $f\in C(X)$ and $h\in C(X,E)$, $fh\in C(X,E)$ is defined by
$fh(x)=f(x) h(x)$ for all $x\in X$.

\begin{defn}{\rm
Let $X$ be a compact Hausdorff space and $E$ be a normed  space
over $\Bbb K$.

(i) A $\Bbb K$-subspace $A(X)$ of $C(X,\Bbb K)$ is called {\em
boundedly normal} if there exists a constant  $M>0$ such that for
any pair $K,K'$ of disjoint closed subsets of $X$ there exists
$f\in A(X)$ with $\|f\|_\infty \le M$, $f=0$ on $K$ and $f=1$ on
$K'$.

(ii) We say that a subspace $A(X,E)$ of $C(X,E)$ is {\em nice} if
there exists a boundedly normal subspace $A(X)$ of $C(X,\BK)$
containing constants such that $A(X) \cdot A(X,E) \subseteq
A(X,E)$, where $A(X) \cdot A(X,E)=\{fh: f\in A(X), h\in A(X,E)\}$.
}
\end{defn}
Here are some examples of nice subspaces of $C(X,E)$.
\begin{exam}\label{exam}
{\rm Let $E$ be a $\Bbb K$-normed space.

{\rm (i)} For a compact metric space $(X,d)$ and $\alpha\in
(0,1]$, let $\Lip_\alpha(X,E)$ be the space of all functions $f: X
\longrightarrow E$ such that \[L(f)=\sup_{x\neq y}
\frac{\|f(x)-f(y)\|}{d^\alpha(x,y)} <\infty.\] Then
$\Lip_\alpha(X,E)$ is a normed space with respect to the norm
$\|f\|_\Lip =\|f\|_\infty+L(f)$, $f\in \Lip_\alpha(X,E)$,
which is complete whenever $E$ is a Banach space. For $\alpha \in
(0,1)$, the closed subspace $\lip_\alpha(X,E)$  of
$\Lip_\alpha(X,E)$ consists of all functions $f\in
\Lip_\alpha(X,E)$ satisfying $ \lim_{d(x,y)\to 0}\frac
{\|f(x)-f(y)\|}{d^\alpha(x,y)}=0$. It is easy to see that  $\Lip_\alpha(X,\Bbb K)$, for $\alpha \in (0,1]$,  and $\lip_\alpha(X, \Bbb K)$, for $\alpha\in (0,1)$,  are  boundedly
normal subspaces of $C(X,\Bbb K)$. Meanwhile, $\Lip_{\alpha}(X,E)$ and $\lip_\alpha(X,E)$ are both nice
subspaces of $C(X,E)$. Moreover, they are Banach algebras if so is
$E$. For the case that $\alpha=1$, we use the notation $\Lip(X,E)$
for $\Lip_{\alpha}(X,E)$.

\vspace*{.15cm}

{\rm (ii)} For $n\in \BN$,  let $C^n([0,1],E)$ be the space
of all continuously $n$-times differentiable functions $f:[0,1]
\longrightarrow E$. Then $C^n([0,1],E)$ is a normed space  with the  norm $\|f\|=\sum_{i=0}^n
\frac{\|f^{(i)}\|_\infty}{i!}$, $f\in C^n([0,1],E)$ which is complete if $E$ is a Banach space.  Also
$C^n([0,1],E)$ is a nice subspace of $C([0,1],E)$. The same is
true for the subspace $\Lip^n([0,1],E)$ of $C^n([0,1],E)$
consisting of all Lipschitz functions  $f:[0,1]\longrightarrow E$ such that
for all $i=1,...,n$, $f^{(i)}\in \Lip([0,1],E)$. We should note
that if $E$ is a Banach algebra, then $C^n([0,1]),E)$ is  a
Banach algebra with respect to the defined norm. Similarly,
$\Lip^n([0,1],E)$ is a Banach algebra with respect to the
following norm
\[\|f\|=\sum_{i=0}^n \frac{\|f^{(i)}\|_\Lip}{i!}\;\;\;\;\;\;\;\; (f\in
\Lip^n([0,1],E).\]

\vspace*{.15cm}

{\rm (iii)} For a compact subset $X$ of the real line,  let ${\rm
AC}(X,E)$ be the space of all absolutely continuous
$E$-valued functions on  $X$. Then ${\rm AC}(X,E)$ is a nice subspace of
$C(X,E)$. Moreover,  $\|f\|=\|f\|_\infty +{\rm
var}(f)$, $f\in {\rm AC}(X,E)$, defines a norm on ${\rm
AC}(X,E)$, where ${\rm var}(f)$ is the total variation of $f\in
{\rm AC}(X,E)$ and ${\rm AC}(X,E)$ is complete whenever $E$ is a Banach space.  If $E$ is a Banach algebra, then so is $({\rm AC}(X,E),
\|\cdot \|)$.
}
\end{exam}

\begin{defn}{\rm  Let $X,Y$ be compact Hausdorff spaces, $E, F$ be $\Bbb K$-normed spaces and
$A(X,E)$ be a subspace of $C(X,E)$. A pair $T,S: A(X,E)
\longrightarrow C(Y,F)$ of additive maps is said to be {\em
jointly separating} if $\coz(f)\cap \coz(g)=\emptyset$ implies
that $\coz(Tf)\cap \coz(Sg)=\emptyset$ for all $f,g\in A(X,E)$.}
\end{defn}

We should note that for any pair of functions $h,k\in C(Y,F)$ we have
$\coz(h)\cap \coz(k)=\emptyset$ if and only if for each $v^*\in
F^*$ and $y\in Y$, $\nu^*(h(y)) \cdot \nu^*(k(y))=0$ if and only
if for each $v^*, w^*\in F^*$ and $y\in Y$, $\nu^*(h(y)) \cdot
w^*(k(y))=0$. Hence a pair of additive maps  $T,S: A(X,E)
\longrightarrow C(Y,F)$ is jointly separating if and only if $(v^*\circ Tf)\cdot (v^*\circ Sg)=0$ holds for all $v^*\in F^*$ whenever $f,g\in A(X,E)$ with $\coz(f)\cap \coz(g)=\emptyset$. This is also equivalent to say that $(v^*\circ Tf)\cdot (w^*\circ Sg)=0$ holds for all $v^*,w^*\in F$ for such functions  $f,g\in A(X,E)$.

Next lemma gives, in particular, the general form of continuous additive jointly separating functionals on boundedly normal subspaces of $C(X)$.

\begin{lem}\label{fun}
Let $X$ be a compact Hausdorff space,  and $\A$ be a boundedly
normal $\Bbb K$-subspace of $C(X,\Bbb K)$ containing constants.
Let $n\in \Bbb N$ and  $\varphi_1,...,\varphi_n: \A
\longrightarrow \Bbb K$ be continuous nonzero additive
functionals on $\A$ such that $\Pi_{j=1}^n f_j=0$ implies that
$\Pi_{j=1}^{n}\, \varphi_j(f_j)=0$ for all $f_1,...f_n\in \A$.
Then there exists $x\in X$  such that $\varphi_j(f)=
\varphi_j(f(x))$ for all $f\in \A$ and $j=1,...,n$. In particular,
in the case that  $\Bbb K=\Bbb C$,  $\varphi_j=\varphi_j(1) \,{\rm Re}
(\delta_x)+\varphi_j(i)\, {\rm Im} (\delta_x)$  for all
$j=1,...,n$ and in the case that $\Bbb K=\Bbb R$,  $\varphi_j=\varphi_j(1)
\,\delta_x$ for all $j=1,...,n$.
\end{lem}

\begin{proof}
We note that, by continuity assumption,  $\varphi_1,...,
\varphi_n$ are real-linear. We first show that the proof can be
reduced to the case $n=2$. Suppose that the lemma has been proven
in this case and choose $u_j\in A$ with $\varphi_j(u_j)\neq 0$,
$j=1,...,n$. Then for any pair of functions $f_1,f_2\in \A$ with
$f_1f_2$=0 we have $f_1f_2 u_3 \cdots u_n=0$.  Then  by
hypothesis, $\varphi_1(f_1)\varphi_2(f_2)= 0$. Hence, by the case
where $n=2$, there exists $x_{1,2}\in X$ such that
$\varphi_1(f)=\varphi_1(f(x_{1,2}))$ and $\varphi_2(f)=\varphi_2(f(x_{1,2}))$,
for all $f\in \A$. A similar argument shows that there exists a
point $x_{2,3}\in X$ such that $\varphi_2(f)=\varphi_2(f(x_{2,3}))$ and
$\varphi_3(f)=\varphi_3(f(x_{2,3}))$ for all $f\in \A$. Thus
\[ \varphi_2(f(x_{1,2}))=\varphi_2(f)=\varphi_2(f(x_{2,3}))
\;\;\;\;\;\;\;(f\in \A).\] We now show that $x_{1,2}=x_{2,3}$. Since
$\varphi_2$ is nonzero, the above equalities show that there
exists a scalar $c\in \BK$ such that $\varphi_2(c)\neq 0$. If
$x_{1,2}\neq x_{2,3}$, then  we can find a function $f\in \A$ such that
$f(x_{1,2})=1$ and $f(x_{2,3})=0$. Then $cf\in \A$ and, using the above
equalities, we have
\[ \varphi_2(c)=\varphi_2(cf(x_{1,2}))=\varphi_2(cf)=\varphi_2(cf(x_{2,3}))=0,\]
a contradiction.  Hence $x_{1,2}=x_{2,3}$.  This argument can be applied for all pairs $\varphi_j,\varphi_k$, where $j,k\in \{1,...,n\}$ are distinct,  to conclude that there exists a point $x\in X$ such that for all $j=1,...,n$
\[ \varphi_j(f)=\varphi_j(f(x))\qquad \qquad  (f\in \A).\]
By the above argument  we can assume without loss of generality
that $n=2$.

First assume that $\Bbb K=\Bbb C$. Then $\A$ is a complex subspace
of $C(X)$ and $\varphi_1,\varphi_2:\A \longrightarrow \BC$ are
nonzero real-linear jointly separating functionals. Hence for $f,g\in \A$ with
disjoint cozero sets, at least one of the equalities
$\varphi_1(f)=0$ and  $\varphi_2(f)=0$ holds. This easily implies
that  for distinct $j,k\in \{1,2\}$, the pair ${\rm
Re}(\varphi_j)$, ${\rm Re}(\varphi_k):\A \longrightarrow \BR$ and
also the pair ${\rm Re}(\varphi_j)$, ${\rm Im}(\varphi_k):\A
\longrightarrow \BR$ of real-linear functionals  on the complex
space $\A$ are jointly separating. Replacing $\varphi_1$ and $\varphi_2$ by $i \varphi_1$ and $i \varphi_2$ if it is necessary,  we may assume that ${\rm Re}(\varphi_1)$ and ${\rm Re}(\varphi_2)$ are both nonzero.
Now we consider the complex-linear functionals $\Phi_1, \Phi_2:\A \longrightarrow \BC$
defined by
\[ \Phi_1(f)={\rm Re}(\varphi_1(f))-i{\rm Re}(\varphi_1(if))\qquad \qquad (f\in
A(X)),\]
 and
\[\Phi_2(f)={\rm Re}(\varphi_2(f))-i{\rm Re}(\varphi_2(if))\qquad \qquad (f\in A(X)).\]
 If $f,g\in \A$ such that $f\cdot g=0$, then $f\cdot
ig=if\cdot g=0=if\cdot ig$ and since ${\rm Re}(\varphi_1)$ and
${\rm Re}(\varphi_2)$ are nonzero jointly separating we deduce that
$\Phi_1,\Phi_2$ are also nonzero jointly separating. We note that
$\Phi_1,\Phi_2$, as  continuous nonzero complex linear functionals
on $\A$,  can be extended to continuous complex-linear functionals
on $C(X)$. Hence there exist (nonzero) complex regular Borel measures $\mu$
and $\nu$ on $X$ satisfying
\[\Phi_1(f)=\int_X f d\mu\;\;\;\; {\rm and} \;\;\; \Phi_2(f)=\int_{X} fd\nu \]
for all $f\in \A$. Using the fact that  $\A$ is boundedly normal
and $\Phi_1,\Phi_2$ are jointly separating,  we can easily deduce
 that for any pair $K_1,K_2$ of disjoint closed subsets  of
$X$ we have $$|\mu|(K_1)\, |\nu|(K_2)=0.$$ We note that
$\supp(\mu)=\supp(\nu)=\{x\}$ for some $x\in X$. Indeed, if there
are distinct points  $x_1 \in \supp(\mu)$ and $x_2 \in
\supp(\nu)$, then choosing  neighborhoods $U_1$ and $U_2$ of $x_1$
and $x_2$, respectively  with disjoint closures, it follows from
the above argument that  $|\mu|(\overline{U_1})\,
|\nu|(\overline{U_2})=0$, which is impossible. This clearly
implies that  $\supp(\mu)=\supp(\nu)=\{x\}$ for some $x\in X$ and consequently  we have  $\Phi_1=\alpha_1 \delta_x$ and $\Phi_2= \alpha_2 \delta_x$
for some $\alpha_1, \alpha_2 \in \Bbb C\backslash \{0\}$. Thus
\[{\rm Re}(\varphi_1)={\rm Re}(\alpha_1 \delta_x)\] and \[{\rm Re}(\varphi_2)={\rm
Re}(\alpha_2 \delta_x).\]
We claim that there are $\beta_1,\beta_2\in \BC$ such that
\begin{align}\label{Im}
{\rm Im}(\varphi_1)={\rm Re}(\beta_1
\delta_x), \; \; {\rm Im}(\varphi_2)={\rm Re}(\beta_2 \delta_x).
\end{align}
We prove the first equality, since the second one is proven in a similar manner. If ${\rm Im}(\varphi_1)=0$  then clearly the desired equality holds for $\beta_1=0$. Hence assume that  ${\rm Im}(\varphi_1)\neq 0$. Then ${\rm Re}(i\varphi_1)\neq 0$ and since ${\rm Re}(\varphi_2)\neq 0$, using the above argument for jointly separating functionals $i\varphi_1$ and $\varphi_2$ we conclude that there exist $x'\in X$ and  scalars  $\lambda_1, \lambda_2\in \BC\backslash\{0\}$ such that
\[{\rm Re}(i \varphi_1)={\rm Re}(\lambda_1 \delta_{x'})\] and \[{\rm Re}(\varphi_2)={\rm
Re}(\lambda_2 \delta_{x'}).\]
Hence ${\rm Re}(\alpha_2 \delta_{x})={\rm Re}(\varphi_2)={\rm
Re}(\lambda_2 \delta_{x'})$. This implies that $x=x'$.  Indeed, if $x\neq x'$ then we can choose $f\in \A$
satisfying $f(x)=0$ and $f(x')=\overline{\lambda_2}$. Then
\[0={\rm Re} (\alpha_2 f(x))={\rm Re} (\varphi_2(f))={\rm Re}(\lambda_2f(x'))=|\lambda_2|^2\neq
0,\] which is a contradiction. Hence $x=x'$ and consequently
\[{\rm Im}(\varphi_1)=-{\rm Re}(i\varphi_1)={\rm Re}(-\lambda_1 \delta_{x}),\] as desired.
Thus (\ref{Im}) hold for some $\beta_1, \beta_2 \in \BC$.
This implies that $\varphi_1(f)= {\rm Re}(\alpha_1 f(x))+i{\rm Re}(\beta_1 f(x))$ for all $f\in \A$. In fact, an easy calculation shows that
\[\varphi_1(f)= \varphi_1(1) {\rm Re}(f(x))+\varphi_1(i)  {\rm Im}(f(x))\qquad \qquad  (f\in
\A).\]  In particular,
$\varphi_1(f)=\varphi_1(f(x))$ holds for all $f\in \A$. Similarly
$\varphi_2= \varphi_2(1)\, {\rm Re}(\delta_x)+\varphi_2(i)\,{\rm
Im}(\delta_x)$, as desired.

 Assume now that $\Bbb K=\Bbb
R$. Then $\A$ is a real subspace of $C_{\Bbb R}(X)$ and
$\varphi_1,\varphi_2: \A \longrightarrow \BR$ are jointly
separating real-linear functionals on $\A$. Extending
$\varphi_1,\varphi_2$ to continuous real-linear functionals on
$C_{\Bbb R}(X)$, the same argument can be applied to show that
$\varphi_j=\varphi_j(1) \delta_{x}$, $j=1,2$, for some point $x\in
X$.
\end{proof}
Let $X,Y$ be compact Hausdorff spaces and $E,F$ be normed spaces over $\Bbb K$.
For a  pair $T,S: A(X,E) \longrightarrow C(Y,F)$ of additive  maps on a subspace $A(X,E)$ of $C(X,E)$ containing constants, we
put
\[ Y_0= \big(\cup_{e\in E} \coz(T(c_e))\big) \cap \big(\cup_{e\in E} \coz(S(c_e))\big), \]
and
\[ Y_c=\{y\in Y_0: \delta_y\circ T, \delta_y\circ S:  A(X,E)\longrightarrow F \; {\rm are} \; \|\cdot\|_\infty {-\rm weak}\; {\rm continuous} \; \}. \]
It is obvious that if there exists at least one  constant function whose images under $T$ and $S$ are constant, then $Y_0=Y$.
We denote the space of all bounded real-linear operators from $E$ to $F$  by $B_{\Bbb R}(E,F)$.
\begin{theorem} Let $X,Y$ be compact Hausdorff spaces, $E,F$ be
normed spaces over $\Bbb K$ and $A(X,E)$ be a nice subspace of
$C(X,E)$ containing constants.  Let $T,S: A(X,E) \longrightarrow
C(Y,F)$ be additive jointly separating maps. Then there exist a
continuous map $\Phi:Y_c \longrightarrow X$ and two families
$\{\Lambda_y\}_{y\in Y_c}$ and $\{\Lambda'_y\}_{y\in Y_c}$ of
real-linear operators from $E$ to $F$ such that
\[ Tf(y)= \Lambda_{y}(f(\Phi(y))) \qquad \qquad (f\in
A(X,E), y\in Y_c),\] and
\[Sf(y)= \Lambda'_{y}(f(\Phi(y))) \qquad \qquad (f\in
A(X,E), y\in Y_c).\] In particular, if $T,S$ are continuous, then
$Y_c=Y_0$ and, moreover, for each $y\in Y_c$, $\Lambda_y\in
B_\BR(E,F)$ and, furthermore,  $y\mapsto \Lambda_y$ is a
continuous map from $Y_c$ to $B_\BR(E,F)$ with respect to the
strong operator topology on $B_\BR(E,F)$.
\end{theorem}

We prove the theorem through the subsequent  lemmas.

In the sequel we assume that $X,Y$ and $E,F$ and also $A(X,E)
\subseteq C(X,E)$ and $T,S: A(X,E)\longrightarrow C(Y,F)$ are as
in the theorem.

By hypotheses there exists a boundedly normal subspace $A(X)$ of
$C(X,\Bbb K)$ such that $A(X)\cdot A(X,E) \subseteq A(X,E)$. For
each $(y, v^*)\in Y\times F^*$ , we put
\[\varphi_{y,v^*}=v^*\circ \delta_y \circ T \; {\rm and} \;\; \psi_{y,v^*}=v^*\circ \delta_y \circ S.\] Then clearly
for each $y\in Y_c$ and each  $v^*, w^* \in F^*$, the maps
$\varphi_{y,v^*}$ and $\psi_{y,w^*}$ are continuous real-linear
functionals on $A(X,E)$ which are jointly separating.

Also for each $(y, v^*,e) \in Y\times F^*\times E$  we  define the
jointly separating additive  functionals $\varphi_{y,v^*}^e,
\psi_{y,v^*}^e: A(X)\longrightarrow \Bbb C$ by
\[\varphi_{y,v^*}^e(f)=\varphi_{y,v^*}(fe)\;\; {\rm and }\;\; \psi_{y,v^*}^e(f) =\psi_{y,v^*}(fe) \qquad \qquad(f\in A(X)).\]
Let $y\in Y_c$. Then $y\in Y_0$ and so there are $e_j\in E$ and
$v_j^* \in F^*$, $j=1,2$, such that $\varphi_{y,v_1^*}^{e_1} \neq
0$ and $\psi_{y,v_2^*}^{e_2} \neq 0$.  Motivated by this,  for
each $y\in Y_c$ we consider the following set
\[H_y=\{ (v_1^*, v_2^*, e_1, e_2) \in F^{*^2}\times E^2:  \varphi_{y,v_1^*}^{e_1} \,{\rm  and} \,
\psi_{y,v_2^*}^{e_2} \, {\rm are\, nonzero}\}. \]

In the next lemmas we assume that $\Bbb K=\Bbb C$. Similar results hold for the case $\Bbb K=\Bbb R$.

\begin{lem} Let $y\in Y_c$. Then there
exists  a unique point $x\in X$ (depending only on $y$) such that for all $(v_1^*, v_2^*, e_1,
e_2) \in H_y$,
\begin{align} \label{1}
\varphi_{y,v_1^*}^{e_1} = \varphi_{y,v_1^*}^{e_1}(1) {\rm Re}(\delta_x)+ \varphi_{y,v_1^*}^{e_1}(i)
{\rm Im}(\delta_x),
\end{align} and
\begin{align}\label{2}
\psi_{y,v_2^*}^{e_2} = \psi_{y,v_2^*}^{e_2}(1) {\rm Re}(\delta_x)+\psi_{y,v_2^*}^{e_2}(i) {\rm
Im}(\delta_x).
\end{align}
\end{lem}
\begin{proof}
For any $(v_1^*, v_2^*, e_1, e_2) \in H_y$,
$\varphi_{y,v_1^*}^{e_1}$ and $\psi_{y,v_2^*}^{e_2}$ are nonzero
jointly separating continuous real-linear functionals on $A(X)$.
Thus by Lemma \ref{fun} there exists a point  $x\in X$ (depending on $v_1^*, v_2^*,y,e_1,e_2$)
such that
\begin{align*}
 \varphi_{y,v_1^*}^{e_1} &= \alpha_1 {\rm Re}(\delta_x)+ \beta_1 {\rm Im}(\delta_x),  \\
 \psi_{y,v_2^*}^{e_2} &= \alpha_2 {\rm Re}(\delta_x)+\beta_2 {\rm Im}(\delta_x),
\end{align*}
where $\alpha_1=\varphi_{y,v_1^*}^{e_1}(1)$, $\beta_1=\varphi_{y,v_1^*}^{e_1}(i)$ and similarly
$\alpha_2=\psi_{y,v_2^*}^{e_2}(1)$ and $\beta_2=\psi_{y,v_2^*}^{e_2}(i)$.

We now show that the point $x$ depends only on $y$.
For this, let $(w_1^*,w_2^*,e'_1,e'_2)$ be another element of $H_y$. Then by
the above argument there exists  $x' \in X$ such that
\[ \varphi_{y,w_1^*}^{e'_1}=  \alpha'_1 \,{\rm Re}(\delta_{x'})+ \beta'_1 \,{\rm Im}(\delta_{x'}),\;\; \;\;
\psi_{y,w_2^*}^{e'_2}=  \alpha'_2 \,{\rm
Re}(\delta_{x'})+\beta'_2 \,{\rm Im}(\delta_{x'}),\]
where $\alpha'_1= \varphi_{y,w_1^*}^{e'_1}(1)$, $\beta'_1=\varphi_{y,w_1^*}^{e'_1}(i)$ and also
$\alpha'_2=\psi_{y,w_2^*}^{e'_2}(1)$ and $\beta'_2=\psi_{y,w_2^*}^{e'_2}(i)$.
 Since $\varphi_{y,v_1^*}^{e_1}$ and $\psi_{y,w_2^*}^{e'_2}$ are also
 nonzero continuous real-linear jointly separating functionals on
$\A(X)$ we also have
 \[ \varphi_{y,v_1^*}^{e_1}= \alpha \,{\rm Re}(\delta_z)+ \beta \,{\rm Im}(\delta_z),\;\; {\rm and} \;\;
\psi_{y,w_2^*}^{e'_2}=\alpha' \,{\rm Re}(\delta_z)+ \beta' \,{\rm
Im}(\delta_z),\] for some $z\in X$ and scalars
$\alpha,\alpha',\beta,\beta'\in \Bbb C$. Hence,
\[ \alpha \,{\rm Re}(\delta_z)+ \beta \,{\rm Im}(\delta_z)=\varphi_{y,v_1^*}^{e_1}=\alpha_1 {\rm Re}(\delta_x)+ \beta_1 {\rm
Im}(\delta_x), \] and, as before, we can conclude that $x=z$.
Similarly $x'=z$, that is  $x=x'$.

It is easy to see that the point $x\in X$ with the desired property is unique.
\end{proof}
 By the above lemma, the point $x\in X$ is the same for all quadruples $(v_1^*,v_2^*, e_1,e_2) \in H_y$.
Hence we can define a map $\Phi: Y_c \longrightarrow X$ which
associates to each point $y\in Y_c$ the unique point $x\in
X$ such that (\ref{1}) and (\ref{2}) hold for all
$(v_1^*,v_2^*,e_1,e_2)\in H_y$.
\begin{lem}
Let  $y\in Y_c$. Then for  all $e\in E$ and $v^*\in F^*$, we have
 \[ \varphi_{y,v^*}^{e}=\varphi_{y,v^*}^{e}(1) \,{\rm
Re}(\delta_{\Phi(y)})+ \varphi_{y,v^*}^{e}(i) \,{\rm
Im}(\delta_{\Phi(y)}),\] and
\[\psi_{y,v^*}^{e}=\psi_{y,v^*}^{e}(1) \,{\rm Re}(\delta_{\Phi(y)})+ \psi_{y,v^*}^{e}(i) \,{\rm
Im}(\delta_{\Phi(y)}).\]
\end{lem}
\begin{proof}
 We prove the first equality, since the
other one is proven similarly.

Suppose that $y\in Y_c$ and let $e\in E$ and $v^*\in F^*$ be
arbitrary. If $\varphi_{y,v^*}^{e}=0$, then
$\varphi_{y,v^*}^{e}(1)=0=\varphi_{y,v^*}^{e}(i)$ and so the
equality is obvious. Hence we assume that $\varphi_{y,v^*}^{e}\neq
0$. Since $y\in Y_c\subseteq Y_0$ it follows easily that there
exist $w^*\in F^*$ and $e'\in E$ such that $\psi_{y,w^*}^{e'}\neq
0$. Therefore, $(v^*,w^*,e,e')\in H_y$ and consequently, by lemma
above, we have again the desired equality.
\end{proof}

Now for each $y\in Y_c$ we define real-linear operators
$\Lambda_y,\Lambda'_y: E \longrightarrow F$,  by
\[ \Lambda_y(e)=T(c_e)(y),\,\,\, {\rm and} \,\,\,\,
\Lambda'_y(e)=S(c_e)(y).\]
Let  $y\in Y_c$, and  $v^*\in E^*$. Put $x=\Phi(y)$. Using the
above lemma and the real-linearity of $\varphi_{y,v^*}$,  for each
$e\in E$ and $f\in A(X)$ we have
\begin{align*}
\varphi_{y,v^*}^{e}(f)&=\varphi_{y,v^*}^{e}(1) \,{\rm
Re}(f(x))+ \varphi_{y,v^*}^{e}(i)\, {\rm Im}(f(x))\\
&= \varphi_{y,v^*}^{e}({\rm Re}(f(x))+i{\rm Im}(f(x)))\\
&=\varphi_{y,v^*}^{e}(f(x)).
\end{align*} Hence for each $y\in Y_c$, $e\in E$ and each $v^*\in F^*$ we have   $\varphi_{y,v^*}(fe)=\varphi_{y,v^*}^{e}(f)=\varphi_{y,v^*}^e(f(\Phi(y)))$, that is
\begin{align}\label{3}
\varphi_{y,v^*}(fe)=v^*(\Lambda_y(f(\Phi(y))e))\qquad
\qquad (f\in A(X)).
\end{align}
Similarly
\begin{align}\label{4}
\psi_{y,v^*}(fe)=v^*(\Lambda'_y(f(\Phi(y))e))\qquad
\qquad (f\in A(X)).
\end{align}

\begin{lem} Let $y\in Y_c$. Then for each $h\in A(X,E)$ with
$h(\Phi(y))=0$ there exists a sequence $\{h_n\}_{n=1}^\infty$ in
$A(X,E)$ such that each $h_n$ vanishes on a neighborhood of
$\Phi(y)$ and $\|h_n-h\|_\infty \to 0$ as $n\to \infty$.
\end{lem}
\begin{proof}
Assume that $x=\Phi(y)$. For each $n\in \BN$ we put
\[U_n=\{z\in X: \|h(z)\|<\frac{1}{n}\}.\]
Let $V_n$ be a neighborhood of $x$ with $\overline{V_n}\subseteq
U_n$. Then $x\in U_n$ for all $n\in \Bbb N$ and, using bounded
normality of $A(X)$, there exist $M>0$ and a sequence $\{g_n\}$ in
$A(X)$ such that $\|g_n\|_\infty<M$, $g_n=1$ on $X\backslash U_n$
and $g_n=0$ on $V_n$, for all $n\in \BN$. Put $h_n=g_n h$, $n\in \BN$. Then $h_n\in A(X,E)$ and, furthermore,
$h_n=0$ on $V_n$ and $\|h-h_n\|_\infty<\frac{1}{n} (1+M)$, that is $h_n \to h$ in $A(X,E)$.
\end{proof}

\begin{lem}
For each $y\in Y_c$ and $v^*\in F^*$
\[ \varphi_{y,v^*}(h)= v^*(\Lambda_y(h(\Phi(y)))),\] and
\[\psi_{y,v^*}(h)=v^*(\Lambda'_y(h(\Phi(y))))\] hold for all $h\in A(X,E)$.
\end{lem}
\begin{proof}
Let $y\in Y_c$ and $v^*\in F^*$. Put $x=\Phi(y)$. To prove the
first equality,  it suffices to show that for each $h\in A(X,E)$
with $h(x)=0$ we have $\varphi_{y,v^*}(h)=0$. Indeed, if we prove
this implication, then for each $h\in A(X,E)$, the element
$e=h(x)$ in $E$ satisfies $(h-c_e)(x)=0$ and consequently
$\varphi_{y,v^*}(h-c_e)=0$ which implies, by (\ref{3}), that
\[ \varphi_{y,v^*}(h)=\varphi_{y,v^*}(c_e)=v^*(\Lambda_{y}(e))=
v^*(\Lambda_{y}(h(x))),\] as desired.

 First consider the case where $h\in A(X,E)$ vanishes  on a neighborhood $U$ of $x$. Using
the normality of $A(X)$, there exists a function $g\in A(X)$ with
$g(x)=1$ and $g=0$ on $X\backslash U$. We note that since $y\in
Y_c\subseteq Y_0$ there exists $e\in E$ and $w^*\in F^*$ such that
$w^*(S(c_e)(y))\neq 0$, that is
$w^*(\Lambda'_{y}(e))=\psi_{y,w^*}^e(1)\neq 0$. Hence for $G=ge\in
A(X,E)$, using (\ref{4}), we have
\[ \psi_{y,w^*}(G)= \psi_{y,w^*}(ge)=w^*(\Lambda'_{y}(g(x)
e))=w^*(\Lambda'_{y}(e))\neq 0.\] Since $\coz(h)\cap
\coz(G)=\emptyset$ we have $\varphi_{y,v^*}(h) \cdot
\psi_{y,v^*}(G)=0$ and being $\psi_{y,v^*}(G)\neq 0$ we get
$\varphi_{y,v^*}(h)=0$, as desired.

Now, the general case that $h\in A(X,E)$ and $h(x)=0$ follows
immediately from the lemma above and the continuity of
$\varphi_{y,v^*}$.
\end{proof}

Let $y\in Y_c$ and $v^*\in F^*$. By the above lemmas for each
$h\in A(X,E)$  we have
\[ v^*(Th(y))=\varphi_{y,v^*}(h)=v^*(\Lambda_{y}(h(\Phi(y))))\]
and \[ v^*(Sh(y))=\psi_{y,v^*}(h)=v^*(\Lambda'_{y}(h(\Phi(y)))).\]
Since these equalities holds for all $v^*\in F^*$ we get
\[Th(y)=\Lambda_y(h(\Phi(y))), \;  \;\;  Sh(y)=\Lambda'_y(h(\Phi(y))) \]
for all $y\in Y_c$ and $h\in A(X,E)$. Hence next lemma completes the proof of the first part of the
theorem.

\begin{lem}
The map $\Phi: Y_c\longrightarrow X$ is continuous.
\end{lem}
\begin{proof}
The proof is straightforward.
\end{proof}

The second part of the theorem is easily verified.

\begin{cor}\label{cor1}
Let $X$ and $Y$ be compact Hausdorff spaces, $E$ be a complex
commutative unital normed algebra and $A(X,E)$ be a nice
subalgebra of $C(X,E)$ containing constants. If $T: A(X,E)
\longrightarrow C(Y)$ is a unital ring homomorphism, then there
exists a subset $Y_c$ of $Y$, a continuous map $\varphi: Y_c
\longrightarrow X$ and a family $(\Lambda_{y})_{y\in Y_c}$ of
real-linear ring homomorphisms such that
\[ Tf(y)= \Lambda_y(f(\varphi(y))) \qquad \qquad (f\in A(X,E), y\in Y_c)\]
Moreover, if $T$ is continuous, then $Y_c=Y$, each $\Lambda_y$ is continuous and $y\to \Lambda_y$ is continuous with respect to the strong operator topology
on $B_\BR(E,\BC)$.
\end{cor}
\begin{proof}
The result follows immediately, since $T$ is easily verified to be a separating map.
\end{proof}
For a unital Banach algebra $E$, let $E^{-1}$ denote the open
subset of invertible elements of $E$ and for $x\in E$ let
$\rho(x)$ denote its spectral radius.

\begin{defn}{\rm
Let $X$ be a compact Hausdorff space and  $E$ be a  unital Banach
algebra. We say that a subalgebra  $\A$ of $C(X,E)$ is {\em
inverse closed} if each element $f\in \A$ with  $f(x)\in E^{-1}$
for all $x\in X$,  is invertible in $\A$, that is the function
$f^{-1}\in C(X,E)$ defined by $f^{-1}(x)=f(x)^{-1}$, $x\in X$, is
an element of $\A$.}
\end{defn}

Clearly, if $\A$ is an inverse closed subalgebra of $C(X,E)$ which
is a Banach algebra with respect to some norm,  then for each
$f\in \A$, \[\rho(f)\le  \sup_{x\in X} \rho(f(x))\le \sup_{x\in X}
\|f(x)\|=\|f\|_\infty.\]

Assume that $X$ is a compact Hausdorff space and $E$ is a
(complex) unital Banach algebra. For each invertible element $f\in
C(X,E)$, and distinct points $s,t\in X$ we have
$f^{-1}(s)-f^{-1}(t)= f^{-1}(s)\big(f(t)-f(s)\big)f^{-1}(t)$.
Using this equality we can easily see that all  Banach algebras
introduced in Example \ref{exam} are inverse closed.

By \cite{Hns}, for a compact metric space $X$ and a commutative
unital (complex) Banach algebra $E$, the maximal ideal space
$M_{\Lip(X,E)}$ of the Banach algebra  $\Lip(X,E)$ is homeomorphic
to $X \times M_E:=\{\delta_x\times \psi: x\in X, \psi\in M_E\} $,
where for each  $x\in X$ and $\psi \in M_E$, $(\delta_x\times
\psi)(f)= \psi (f(x))$ for all $f\in \Lip(X,E)$. Next corollary,
in particular,  gives a similar result for $\|\cdot\|$-continuous
complex ring homomorphisms on ${\rm Lip}(X,E)$.

\begin{cor}\label{cor2}
Let $X$ be a compact Hausdorff  space,  $E$ be a complex
commutative unital Banach algebra and $\A$ be an inverse closed
subalgebra of $C(X,E)$ which is a Banach algebra with respect to a
norm  $\|\cdot\|$ with $\|\cdot \|\ge \|\cdot\|_\infty$. Assume
that $\varphi: \A \longrightarrow \BC$ is a nonzero ring
homomorphism. Then

{\rm (i)} $\varphi$ is $\|\cdot \|_\infty$-continuous if and only if it is continuous with respect to $\|\cdot\|$.

{\rm (ii)} If $\varphi$ is $\|\cdot \|$-continuous, then there exists a point $x\in X$ and a continuous ring homomorphism $\Lambda: E \longrightarrow \Bbb C$ such that
\[\varphi(f)=\Lambda(f(x)) \qquad \qquad (f\in \A).\]
\end{cor}
\begin{proof}
(i) For the if part, assume that $\varphi$ is continuous with
respect to the norm $\|\cdot\|$. Then there exists $K>0$ such that
$|\varphi(f)|\le K \|f\|$ for all $f\in \A$. Hence for all $n\in
\BN$
\[ |\varphi(f)|^n=|\varphi(f^n)|\le K \|f^n\| \qquad \qquad (f\in \A),\]
which implies that $|\varphi(f)|\le K^{\frac{1}{n}} \|f^n\|^
{\frac{1}{n}}$ for all $f\in \A$. Tending $n$ to infinity, we get
$|\varphi(f)|\le \rho(f)\le \|f\|_\infty$, that is $\varphi$ is
$\|\cdot\|_\infty$-continuous.

The only if part is trivial.

(ii) We note that the ring homomorphism $\varphi$ is unital, since
it is nonzero.  Hence, the result is immediate from Corollary
\ref{cor1}.
\end{proof}
\begin{cor}\label{cor3}
Let $X$ be a compact metric space, $K$ be a compact Hausdorff
space and $\varphi: \Lip(X,C(K)))  \longrightarrow \Bbb C$ be a
nonzero continuous ring homomorphism. Then there exist  points
$x\in X$ and $t\in K$ such that either
\[\varphi(f)= f(x)(t)\qquad \qquad (f\in \Lip(X,C(K)))\]
or
\[\varphi(f)=\overline{f(x)(t)} \qquad \qquad (f\in \Lip(X,C(K))).\]
\end{cor}
\begin{proof}
It is obvious that $\varphi(1)=1$. Hence, using Corollary
\ref{cor2}, we conclude that there exist a point $x\in X$ and a
continuous ring homomorphism $\Lambda: C(K) \longrightarrow \Bbb
C$ such that
\[\varphi(f)= \Lambda(f(x))\;\;\;\; (f\in \Lip(X, C(K)).\]
Using Corollary \ref{cor2} once again, we can find a point $t\in
K$ and a continuous ring homomorphism $\tau: \Bbb C
\longrightarrow \Bbb C$ such that
\[\Lambda(h)=\tau(h(t)) \qquad \qquad (h\in C(K)).\]
Since $z\mapsto z$ and $z\mapsto \overline{z}$ are the only
continuous ring homomorphisms on $\Bbb C$ we get the desired
description for $\varphi$.
\end{proof}

\begin{remark}{\rm Let $E$ be a commutative unital Banach algebra. As we noted before
all Banach algebras $\lip_\alpha(X,E)$, for a compact metric space $X$ and $\alpha\in (0,1)$,
$C^n([0,1],E)$, $\Lip^n([0,1],E)$ and $AC(X,E)$, for a compact subset $X$ of the real line,
are inverse closed. Since their norms also satisfies $\|\cdot\|\ge \|\cdot\|_\infty$,  the above
corollary holds true whenever $\Lip(X,C(K))$ is replaced by any of these Banach algebras for $E=C(K)$.
}
\end{remark}

In \cite{bot}, Botelho and Jamison gave the  description  of
unital homomorphisms $T: \Lip(X,E) \longrightarrow \Lip(Y,E)$ for
compact metric spaces $X$ and $Y$ where $Y$ is connected and the
target (Banach) algebra $E$ is either $c_0=C_0(\Bbb N)$ or
$\l_\infty=C(\beta \Bbb N)$. We note that semisimplicity of $E$
implies that $\Lip(X,E)$ is semisimple, as well. Hence such a
homomorphism $T$ is continuous with respect to the Lipschitz norms
on both sides. In the next theorem (see also Remark \ref{rem}) we
give a generalization of \cite{bot} for continuous ring
homomorphisms between some subalgebras of $C(X,E)$,
where $X$  is compact and $E$ is a certain subalgebra of $C(K)$ for an
arbitrary compact Hausorff space $K$.

For a (complex) Banach algebra $A$, let $\mathcal{R}(A)$ denote
the set of all nonzero continuous ring homomorphisms $\varphi: A \longrightarrow \BC$.

\begin{theorem}\label{conn}
Let $X$ and $Y$ be compact metric spaces such that $Y$ is
connected. Let $K$ and $K'$ be  compact Hausdorff spaces  and $T:
\Lip(X,C(K)) \longrightarrow \Lip(Y,C(K'))$ be a unital ring
homomorphism. If $T$ is continuous with respect to the Lipschitz
norms on both sides, then there are a continuous function $\tau:
K'\longrightarrow K$, a family  $\{\psi_s\}_{s\in K'}$ of continuous functions from $Y$ to $X$ and a clopen subset $\mathcal{C}$ of $K'$ such that
\[ T(f)(y)(s)=\left\lbrace
  \begin{array}{c l}
    f(\psi_s(y))(\tau(s))   &  \,\, s\in \mathcal{C}\\
    \\
    \overline{f(\psi_s(y))(\tau(s))}&  \,\, s\in K'\backslash \mathcal{C}\,\,
  \end{array}
  \right . \]
  for all $f\in \Lip(X,C(K))$  and $y\in Y$.
\end{theorem}
\begin{proof}  Since $T$ is a continuous unital ring homomorphism, it follows that for each $\Psi\in \mathcal{R}(\Lip(Y,C(K')))$,
the map $\Psi\circ T: \Lip(X,C(K)) \longrightarrow \Bbb C$  is an
element of $\mathcal{R}(\Lip(X,C(K)))$. Hence, by Corollary
\ref{cor3} there exist  points $x\in X$ and $t\in K$ such that
either
\[\Psi(Tf)=f(x)(t) \qquad \qquad(f\in \Lip(X,C(K)))\]
or
\[\Psi(Tf)=\overline{f(x)(t)}\qquad \qquad (f\in \Lip(X,C(K))).\]
Now for each $y\in Y$ and $s\in K'$, the map $\Psi: \Lip(Y, C(K'))
\longrightarrow \Bbb C$ defined by $\Psi(g)=g(y)(s)$, $g\in
\Lip(Y,C(K'))$, is an element of  $\mathcal{R}(\Lip(Y,C(K')))$, and so
there are points $x\in X$ and $t\in K$ satisfying one of the above
equalities. Thus  we can define a map $\sigma: Y\times K'
\longrightarrow X \times K$ by $\sigma(y, s)= (x,t)$, such that
either
\[ T(f)(y)(s)=f(x)(t)\qquad \qquad (f\in \Lip(X,C(K)))\]
or
\[ T(f)(y)(s)=\overline{f(x)(t)}\qquad \qquad (f\in \Lip(X,C(K))).\]
Clearly $\sigma$ is continuous with respect to the product
topology.  Now we claim that  the point $t\in K$ in the above
equalities is independent of the choice of $y\in Y$, that is, it
 depends only on the point $s\in K'$.

First we note that $-1=T(-1)=T(i^2)=(Ti)^2$. Hence, for each $y\in
Y$ and $s\in K'$, $Ti(y)(s)\in \{i,-i\}$. For each $s\in K'$ we
put
\[A_s=\{y\in Y: Ti(y)(s)=i\}.\]
Obviously, $A_s$ is a clopen subset of $Y$ and since $Y$ is
connected we conclude that for each $s\in K'$ we have either
$A_s=Y$ or $A_s=\emptyset$. Thus for each $s\in K'$ we have either

{\bf Case I.} $Ti(y)(s)=i$ for all $y\in Y$, or

{\bf Case II.} $Ti(y)(s)=-i$ for all $y\in Y$.

Let $\pi_2: X\times K\longrightarrow K$ be defined by
$\pi_2(z,u)=u$ for $(z,u)\in X \times K$. To prove the claim, fix
an element $s\in K'$. Assume, furthermore, that Case I holds for
$s$. Then for each $y\in Y$, by the above description of $T$, we
have
\[Tf(y)(s)=f(x)(t) \qquad \qquad  (f\in \Lip(X,C(K))),\]
where $(x,t)=\sigma(y,s)$. Choose an arbitrary $y\in Y$ and put
$t=\pi_2(\sigma(y,s))$. Since $Y$ is connected, it suffices to
show that the open subset  $U_s=\{z\in Y: \pi_2(\sigma(z,s))\neq
t\}$ of $Y$ is  closed. For this, assume that $\{y_n\}$ is  a
sequence in $U_s$ converging to a point $y_0\in Y\backslash U_s$.
Put $(x_0,t_0)=\sigma(y_0,s)$ and for each $n\in \BN$, put
$(x_n,t_n)=\sigma(y_n,s)$. Then by hypothesis we have $t_0=t$ and
$t_n\neq t$ for all $n\in \BN$. For each $\lambda \in C(K)$, let
$c_\lambda: X \longrightarrow C(K)$ be the constant function
defined by $c_\lambda(x)=\lambda$, $x\in X$. Then  we have
\[T(c_\lambda)(y_0)(s)=\lambda(t)\]  and
\[T(c_\lambda)(y_n)(s)=\lambda(t_n)\] for all $\lambda\in C(K)$ and $n\in \BN$.
Therefore,
\begin{align*}
|\lambda(t_n)-\lambda(t)|
&=|T(c_\lambda)(y_n)(s)-T(c_\lambda)(y_0)(s)| \le
\|T(c_\lambda)(y_n)-
T(c_\lambda)(y_0)\|_\infty \\
&\le L(T(c_\lambda)) d(y_n,y_0)  \le \|T(c_\lambda))\| d(y_n,y_0)
\end{align*}
Since  $T$ is continuous we have
\begin{equation}\label{6}
|\lambda(t_n)-\lambda(t)|\le \|T\|\, \|\lambda\|_\infty d(y_n,y_0)
\end{equation}
 for all $\lambda \in C(K)$. This shows that
\[ \|\delta_{t_n}-\delta_t\|=\sup_{\|\lambda\|_\infty \le 1}  |\lambda(t_n)-\lambda(t)|\le \|T\|\,
d(y_n,y_0), \] where $\delta_{u}: C(K)\longrightarrow \Bbb C$ is
the evaluation functional at a point $u\in K$.    We note that for
any pair of distinct points $u,v\in K$, we have
$\|\delta_{u}-\delta_v\|=2$. Then, by the above inequality we have
$d(y_n,y_0)\ge \frac{2}{\|T\|}$ for all $n\in \BN$, which is
impossible. This argument shows that $U_s$ is a clopen subset of
$Y$. Being $Y$ connected we get $U_s=\emptyset $, which proves our
claim in Case I.

Assume now that Case II holds for the given $s\in K'$. Then for
each $y\in Y$ we have
\[Tf(y)(s)=\overline{f(x)(t)} \qquad \qquad  (f\in \Lip(X,C(K))),\]
where $(x,t)=\sigma(y,s)$. Using the same argument, as in the
previous case, we can show that the open subset  $U_s=\{z\in Y:
\pi_2(\sigma(z,s))\neq t\}$ of $Y$ is  closed. Hence
$U_s=\emptyset $, which proves our claim in Case II.

By the above argument we can define a function
$\tau:K'\longrightarrow K$ which associates to each point $s\in k'$, the unique point $t\in K$ satisfying
either
\[ T(f)(y)(s)= f(\psi_s(y))(\tau(s))\qquad \qquad (f\in \Lip(X,C(K)), y\in Y)\]
or
\[T(f)(y)(s)= \overline{f(\psi_s(y))(\tau(s))}\qquad \qquad (f\in \Lip(X,C(K)), y\in Y)\]
where, $\psi_s(y)= \pi_1(\sigma(y,s))$. It is obvious that $\psi_s:Y \longrightarrow X$ is continuous for each $s\in K'$.

To conclude the theorem,  we put
\[\mathcal{C}=\{s\in K': Ti(y)=i\,\, {\rm for \,\, all\,\, } y\in Y\}.\]
Then, by the Cases I,II we have
\[K'\backslash \mathcal{C}=\{s\in K': Ti(y)=-i \, \, {\rm for \,\, all\,\, } y\in Y\},\]
that is  $\mathcal{C}$ is a clopen subset of $K'$. Moreover, the above description of $T$  shows that if $s\in \mathcal{C}$, then  $Tf(y)(s)=f(\psi_s(y))(\tau(s))$ for all $f\in \Lip(X,C(K))$ and $y\in Y$, and  if $s\notin \mathcal{C}$, then $Tf(y)(s)=\overline{f(\psi_s(y))(\tau(s))}$ for all $f\in \Lip(X,C(K))$ and $y\in Y$, as desired.
\end{proof}
\begin{theorem}
Let $K$ and $K'$ be  compact Hausdorff spaces, and $T:\A \longrightarrow \B$ be a continuous unital
ring homomorphism where $\A$ is either $C^n([0,1],C(K))$ or $\Lip^n([0,1],C(K))$ and, similarly
$\B$ is either $C^n([0,1],C(K'))$ or $\Lip^n([0,1],C(K'))$ for some $n\in \Bbb N$.  Then there are a continuous function $\tau:
K'\longrightarrow K$, a family $\{\psi_s\}_{s\in K'}$ of continuous functions on $[0,1]$ and a clopen subset $\mathcal{C}$ of $K'$  such that

\[ T(f)(y)(s)=\left\lbrace
  \begin{array}{c l}
    f(\psi_s(y))(\tau(s))   &  \,\, s\in \mathcal{C}\\
    \\
    \overline{f(\psi_s(y))(\tau(s))}&  \,\, s\in K'\backslash \mathcal{C}\,\,
  \end{array}
  \right . \]
  for all $f\in A$  and $y\in [0,1]$.
\end{theorem}
\begin{proof}
The proof is basically the same proof as in Theorem \ref{conn}. It
suffices to note that for the inequality (\ref{6}) we
may use the vector-valued mean value theorem to conclude that
\begin{align*}
|\lambda(t_n)-\lambda(t)|
&=|T(c_\lambda)(y_n)(s)-T(c_\lambda)(y_0)(s)| \le
\|T(c_\lambda)(y_n)-
T(c_\lambda)(y_0)\|_\infty \\
&\le \|T(c_\lambda)'\|_\infty |y_n-y_0|\le \|T(c_\lambda)\| \,
|y_n-y|.
\end{align*}
Hence we have again
\[|\lambda(t_n)-\lambda(t)|\le \|T\|\, \|\lambda\|_\infty |y_n-y_0|\]
 for all $\lambda \in C(K)$.
\end{proof}

\begin{remark} \label{rem}
 {\rm
 The proof of the above theorems work if $C(K)$ and $C(K')$ are replaced by  natural uniform algebras $A$ and $B$
  on $K$ and $K'$, respectively, such that  for each $t\in K$, $\delta_t$ is an isolated point of $K$ with respect to the operator norm, that is for each $t\in K$ there exists $c_t>0$ such that $\|\delta_t-\delta_s\| >c_t$ holds for all points $s\in K$ distinct from $t$.  Here $\|\delta_t-\delta_s\|$ denotes the operator norm of $\delta_t-\delta_s$ on $A$. In particular, if  $A$ is a natural  uniform algebra on a compact Hausdorff space $K$ and $M>0$ such that for  all distinct points $s,t\in K$ there exists a function $f\in A$ with $f(s)=0$, $f(t)=1$ and $\|f\|_\infty\le M$, then $A$ has the desired property.
}
\end{remark}

\end{document}